\documentclass[epsfig,latexsym,amsfonts,twoside]{article}
\usepackage{amssymb,enumerate}
\usepackage{amsmath,amscd}

\def\part#1{\frac{\partial\phantom{#1}}{\partial#1}}
\newtheorem{thm}{Theorem}

\newtheorem{proposition}[thm]{Proposition}
\newtheorem{lemma}[thm]{Lemma}
\newtheorem{corollary}[thm]{Corollary}

\newenvironment{proof}{\begin{trivlist}\item[]{\bf Proof} }%
{\hfill $\Box$ \end{trivlist}}
\newenvironment{definition}{\begin{trivlist}\item[]{\bf Definition}\em }%
{\end{trivlist}}
\newenvironment{remark}{\begin{trivlist}\item[]{\bf Remark} }%
{\end{trivlist}}
\newenvironment{example}{\begin{trivlist}\item[]{\bf Example} }%
{\end{trivlist}}
\newenvironment{question}{\begin{trivlist}\item[]{\bf Question} }%
{\end{trivlist}}

\def\Z{\ifmmode{{\mathbb Z}}\else{${\mathbb Z}$}\fi}
\def\Q{\ifmmode{{\mathbb Q}}\else{${\mathbb Q}$}\fi}
\def\C{\ifmmode{{\mathbb C}}\else{${\mathbb C}$}\fi}
\def\P{\ifmmode{{\mathbb P}}\else{${\mathbb P}$}\fi}

\def\H{\ifmmode{{\mathrm H}}\else{${\mathrm H}$}\fi}

\def\B{\ifmmode{{\mathcal B}}\else{${\mathcal B}$}\fi}
\def\E{\ifmmode{{\mathcal E}}\else{${\mathcal E}$}\fi}
\def\F{\ifmmode{{\mathcal F}}\else{${\mathcal F}$}\fi}
\def\K{\ifmmode{{\mathcal K}}\else{${\mathcal K}$}\fi}
\def\L{\ifmmode{{\mathcal L}}\else{${\mathcal L}$}\fi}
\def\M{\ifmmode{{\mathcal M}}\else{${\mathcal M}$}\fi}
\def\N{\ifmmode{{\mathcal N}}\else{${\mathcal N}$}\fi}
\def\O{\ifmmode{{\mathcal O}}\else{${\mathcal O}$}\fi}
\def\U{\ifmmode{{\mathcal U}}\else{${\mathcal U}$}\fi}
\def\V{\ifmmode{{\mathcal V}}\else{${\mathcal V}$}\fi}
\def\X{\ifmmode{{\mathcal X}}\else{${\mathcal X}$}\fi}

\def\Br{\ifmmode{{\mathrm{Br}}}\else{${\mathrm{Br}}$}\fi}
\def\OG{\ifmmode{\widetilde{\cal M}_4}\else{$\widetilde{\cal M}_4$}\fi}
\def\D{\ifmmode{{\mathcal D}^b}\else{${{\mathcal
    D}^b}$}\fi}
\def\Shah{\ifmmode{\amalg\hspace*{-3.5pt}\amalg}\else{$\amalg\hspace*{-3.5pt}\amalg$}\fi}

\begin{document}

\title{Isotrivial elliptic K3 surfaces and Lagrangian fibrations}
\author{Justin Sawon}
\date{May, 2014}
\maketitle

\begin{abstract}
A fibration is said to be {\em isotrivial\/} if all of its smooth fibres are isomorphic to a single fixed variety. We classify the elliptic K3 surfaces that are isotrivial, and use them to construct Lagrangian fibrations that are isotrivial. We then modify the construction to produce new examples of holomorphic symplectic orbifolds, that also admit isotrivial Lagrangian fibrations.
\end{abstract}

\maketitle

\section{Introduction}

The purpose of this article is to investigate isotrivial Lagrangian fibrations. A {\em Lagrangian fibration\/} is a fibration $\pi:X\rightarrow\P^n$ of a $2n$-dimensional holomorphic symplectic manifold $X$ whose smooth fibres are $n$-dimensional abelian varieties that are Lagrangian with respect to the holomorphic symplectic structure. In~\cite{sawon12} the author proved that there are finitely many Lagrangian fibrations up to deformation, under some additional natural hypotheses. One of these hypotheses was that the fibration is a maximal variation of abelian varieties; at the other extreme we have isotrivial fibrations, where the fibres do not vary at all.

\begin{definition}
A fibration $\pi:X\rightarrow B$ is said to be isotrivial if all of its smooth fibres are isomorphic to a single fixed variety $F$.
\end{definition}

Denote by $\Delta\subset B$ the discriminant locus, which parameterizes singular fibres. Each point $b\in B\backslash\Delta$ will be contained in a small neighbourhood $U$ (in the analytic topology) over which the fibration is trivial, isomorphic to $U\times F$; it suffices to choose $U$ to be simply connected, for then there can be no monodromy.

For an elliptic fibration, being isotrivial is the same as having constant $j$-function. We begin by studying isotrivial elliptic K3 surfaces. Some of these arise as surfaces of Kummer type, i.e., resolutions of orbifolds $T/G$ where $G$ is a finite subgroup of $\mathrm{SU}(2)$ acting on a two-dimensional complex torus $T$. However, to find all isotrivial elliptic K3 surfaces we need to use Weierstrass models.

In higher dimensions, we produce isotrivial Lagrangian fibrations by taking the Hilbert schemes of isotrivial elliptic K3 surfaces. Again, some of these arise as resolutions of orbifolds $T/\Gamma$ where $\Gamma$ is now a finite subgroup of $\mathrm{Sp}(n)$ acting on a $2n$-dimensional complex torus $T$. We then describe two methods for constructing new examples: modifying the $\Gamma$-action on $T$ and restricting the action to a subgroup $\Gamma^{\prime}<\Gamma$ (the latter idea comes from Matsushita~\cite{matsushita01}). Using these methods, we succeed in producing new holomorphic symplectic orbifolds $X$, with isotrivial Lagrangian fibrations $\pi:X\rightarrow\P^n$, in all dimensions six and greater. Unfortunately none of our examples admits a symplectic desingularization.

The author would like to thank Remke Kloosterman, Matthias Sch{\"u}tt, Ivan Smith, and Misha Verbitsky for helpful suggestions. The author gratefully acknowledges support from the NSF, grant number DMS-1206309.

\section{Elliptic K3 surfaces}

\subsection{Singular fibres}

A general reference for elliptic fibrations on surfaces is Barth et al.~\cite{bhpv04}. An isotrivial elliptic fibration will have constant $j$-function. Suppose we have an isotrivial elliptic fibration over the unit disc, with a singular fibre above $0$. By the Stable Reduction Theorem, a base change of finite order will produce a fibration whose fibre above $0$ is of Kodaira type $I_0$ (a smooth elliptic curve) or $I_k$ (a cycle of $k\geq 1$ rational curves). Since the latter would correspond to a pole of the $j$-function, it cannot occur if the $j$-function is constant. Therefore Stable Reduction produces a smooth fibration, indeed a trivial fibration, and the monodromy of the original fibration must be of finite order. The table below lists all Kodaira types of singular fibres with monodromy of finite order. Note that the monodromy is only determined up to conjugation by an element of $\mathrm{SL}(2,\mathbb{Z})$.

\begin{center}
\begin{tabular}{|c|c|c|c|c|}
  \hline
  Kodaira type & Dynkin diagram & Euler number & monodromy & order \\
  \hline
  $II$ cusp & $\tilde{A}_0$ & 2 & $\alpha=\left(\begin{array}{cc} 1 & 1 \\ -1 & 0 \end{array}\right)$ & 6 \\
  $III$ tacnode & $\tilde{A}_1$ & 3 & $\beta=\left(\begin{array}{cc} 0 & 1 \\ -1 & 0 \end{array}\right)$ & 4 \\
  $IV$ triple point & $\tilde{A}_2$ & 4 & $\alpha^2=\left(\begin{array}{cc} 0 & 1 \\ -1 & -1\end{array}\right)$ & 3 \\
  $I^*_0$ & $\tilde{D}_4$ & 6 & $\alpha^3=\beta^2=\left(\begin{array}{cc} -1 & 0 \\ 0 & -1 \end{array}\right)$ & 2 \\
  $II^*$ & $\tilde{E}_8$ & 10 & $\alpha^5=\left(\begin{array}{cc} 0 & -1 \\ 1 & 1 \end{array}\right)$ & 6 \\
  $III^*$ & $\tilde{E}_7$ & 9 & $\beta^3=\left(\begin{array}{cc} 0 & -1 \\ 1 & 0 \end{array}\right)$ & 4 \\
  $IV^*$ & $\tilde{E}_6$ & 8 & $\alpha^4=\left(\begin{array}{cc} -1 & -1 \\ 1 & 0 \end{array}\right)$ & 3 \\
  \hline
\end{tabular}
\end{center}

\subsection{K3 surfaces of Kummer type}

The `classical' Kummer surface is given by quotienting an abelian surface by the group $\mathbb{Z}/2\mathbb{Z}$ acting by multiplication by $-1$, and resolving the sixteen singular points. Similar examples can be obtained from cyclic groups of orders 3, 4, and 6. These examples all admit isotrivial elliptic fibrations.

\begin{example}
Let $E$ be an elliptic curve and construct a Kummer surface $S$ from $E\times E$. Recall that $S$ is obtained by blowing up the sixteen $A_1$ singularities of $E\times E/(\mathbb{Z}/2\mathbb{Z})$. Projecting onto, say, the second factor induces a morphism
$$S\longrightarrow E\times E/(\mathbb{Z}/2\mathbb{Z})\longrightarrow E/(\mathbb{Z}/2\mathbb{Z})\cong\P^1$$
which makes $S$ into an isotrivial elliptic fibration. This fibration has four singular fibres, sitting above the branch points of $E\rightarrow\P^1$, and each singular fibre consists of a central
$$E/(\mathbb{Z}/2\mathbb{Z})\cong\P^1$$
with multiplicity two plus four additional $\P^1$s touching the central $\P^1$ at the four branch points (these additional $\P^1$s come from the resolutions of the sixteen $A_1$ singularities). In other words, $S\rightarrow\P^1$ has four singular fibres of type $I^*_0$. Note that each singular fibre has Euler number 6, so the Euler number of $S$ will be $4\times 6=24$ as required.
\end{example}

\begin{remark}
More generally, $E\times E$ can be replaced by any two-dimensional complex torus $A$ that is fibred over an elliptic curve $E$. If $E\cong\C/\langle 1,\tau\rangle$, then such a torus can be written
$$A\cong\mathbb{C}^2/\langle (1,0),(\sigma,0),(a,1),(b,\tau)\rangle,$$
and projection onto the second coordinate induces $A\rightarrow E$, with all fibres isomorphic to $F\cong\mathbb{C}/\langle 1,\sigma\rangle$. In fact, there is an isomorphism
\begin{eqnarray*}
\mathbb{C}^2/\langle (1,0),(\sigma,0),(a,1),(b,\tau)\rangle & \longrightarrow & \mathbb{C}^2/\langle (1,0),(\sigma,0),(0,1),(b-a\tau,\tau)\rangle \\
(z,w) & \longmapsto & (z-aw,w),
\end{eqnarray*}
so we can normalize by assuming that $a=0$. Then $A$ is clearly isomorphic to $F\times E$ if and only if $b\equiv 0\;(\mathrm{mod}\; 1,\sigma)$.
\end{remark}

\begin{example}
Let $\zeta=\frac{1+\sqrt{3}i}{2}$ be a primitive sixth root of unity and let $E$ be the elliptic curve $\mathbb{C}/\langle 1,\zeta\rangle$, which has automorphism group of order six generated by multiplication by $\zeta$. The group
$$G:=\left\langle\left(\begin{array}{cc} \zeta^2 & 0 \\ 0 & \zeta^{-2} \end{array}\right)\right\rangle$$
of order three acts on $E\times E$ with nine fixed points. If we resolve the nine $A_2$ singularities of the quotient $E\times E/G$ we obtain a K3 surface $S$. There is a morphism
$$S\longrightarrow E\times E/G\longrightarrow E/\langle\zeta^2\rangle\cong\P^1$$
which makes $S$ into an isotrivial elliptic fibration. This fibration has three singular fibres, sitting above the branch points of $E\rightarrow\P^1$, and each singular fibre consists of a central
$$E/\langle\zeta^2\rangle\cong\P^1$$
with multiplicity three plus three additional chains of two $\P^1$s touching the central $\P^1$ at the three branch points (these additional chains of $\P^1$s come from the resolutions of the nine $A_2$ singularities). In other words, $S\rightarrow\P^1$ has three singular fibres of type $IV^*$. Note that each singular fibre has Euler number 8, so the Euler number of $S$ will be $3\times 8=24$ as required.
\end{example}

\begin{remark}
One could try to generalize this example by replacing $E\times E$ by a two-dimensional complex torus $A\cong\mathbb{C}^2/\langle (1,0),(\zeta,0),(0,1),(b,\zeta)\rangle,$ fibered over $E$ with fibres isomorphic to $E$. However, there exists a $G$-action only if $\zeta b\equiv 0\;(\mathrm{mod}\; 1,\zeta)$, which implies $A\cong E\times E$. (The same observation applies to the next two examples.)
\end{remark}

\begin{example}
Let $E$ be the elliptic curve $\mathbb{C}/\langle 1,i\rangle$, which has automorphism group of order four generated by multiplication by $i$. The group
$$G:=\left\langle\left(\begin{array}{cc} i & 0 \\ 0 & -i \end{array}\right)\right\rangle$$
of order four acts on $E\times E$ with four fixed points and twelve points with isotropy subgroup $\langle\pm I\rangle$. The quotient $E\times E/G$ therefore has four $A_3$ singularities and six $A_1$ singularities; resolving them yields a K3 surface $S$. There is a morphism
$$S\longrightarrow E\times E/G\longrightarrow E/\langle i\rangle\cong\P^1$$
which makes $S$ into an isotrivial elliptic fibration. This fibration has three singular fibres, sitting above the branch points of $E\rightarrow\P^1$. Two of these singular fibres consist of a central
$$E/\langle i\rangle\cong\P^1$$
with multiplicity four plus two additional chains of three $\P^1$s and one additional $\P^1$ touching the central $\P^1$ at the three branch points (coming from the resolutions of the four $A_3$ singularities and two of the $A_1$ singularities); these singular fibres are of type $III^*$. The third singular fibre is of type $I^*_0$, as in the first example, and accounts for the resolutions of the remaining four $A_1$ singularities. The singular fibres of type $III^*$ have Euler number 9, and the singular fibre of type $I^*_0$ has Euler number 6, so the Euler number of $S$ will be $2\times 9+6=24$ as required.
\end{example}

\begin{example}
Let $\zeta=\frac{1+\sqrt{3}i}{2}$ be a primitive sixth root of unity and let $E$ be the elliptic curve $\mathbb{C}/\langle 1,\zeta\rangle$, which has automorphism group of order six generated by multiplication by $\zeta$. The group
$$G:=\left\langle\left(\begin{array}{cc} \zeta & 0 \\ 0 & \zeta^{-1} \end{array}\right)\right\rangle$$
of order six acts on $E\times E$ with one fixed point. The subgroups
$$\left\langle\left(\begin{array}{cc} \zeta^2 & 0 \\ 0 & \zeta^{-2} \end{array}\right)\right\rangle\qquad\qquad\mbox{and}\qquad\qquad\left\langle\left(\begin{array}{cc} -1 & 0 \\ 0 & -1 \end{array}\right)\right\rangle$$
fix eight and fifteen additional points, respectively. The quotient $E\times E/G$ therefore has one $A_5$ singularity, four $A_3$ singularities, and five $A_1$ singularities; resolving them yields a K3 surface $S$. There is a morphism
$$S\longrightarrow E\times E/G\longrightarrow E/\langle\zeta\rangle\cong\P^1$$
which makes $S$ into an isotrivial elliptic fibration. This fibration has three singular fibres, sitting above the branch points of $E\rightarrow\P^1$. One of these singular fibres consists of a central
$$E/\langle\zeta\rangle\cong\P^1$$
with multiplicity six plus an additional chain of five $\P^1$s, an additional chain of three $\P^1$s, and one additional $\P^1$ touching the central $\P^1$ at the three branch points (coming from the resolutions of the $A_5$ singularity, one of the $A_3$ singularities, and one of the $A_1$ singularities); this singular fibre is of type $II^*$. A second singular fibre is of type $IV^*$, as in the second example, and accounts for the resolutions of the remaining three $A_3$ singularities. The third singular fibre is of type $I^*_0$, as in the first example, and accounts for the resolutions of the remaining four $A_1$ singularities. The singular fibre of type $II^*$ has Euler number 10, the singular fibre of type $IV^*$ has Euler number 8, and the singular fibre of type $I^*_0$ has Euler number 6, so the Euler number of $S$ will be $10+8+6=24$ as required.
\end{example}

\begin{remark}
The singular fibres in the above examples are all of starred type ($I^*_0$, $II^*$, $III^*$, $IV^*$); they appear after quotienting by a finite group and then blowing-up singular points. Singular fibres of non-starred type do not appear in K3 surfaces of Kummer type because blow-downs are also required to produce them~\cite{bhpv04}; we will deal with them in the next subsection. Note that the above examples exhibit all possible configurations of singular fibres of starred type for elliptic K3 surfaces. This follows from an Euler number calculation, as
$$6+6+6+6,\qquad 8+8+8,\qquad 9+9+6,\qquad\mbox{and}\qquad 10+8+6$$
are the only partitions of 24 into summands of 6, 8, 9, and 10. By contrast, Miranda and Persson~\cite{mp89} classified the configurations of semistable (i.e., type $I_n$) singular fibres that can occur in a non-isotrivial elliptic K3 surface, and there are hundreds of possibilities.
\end{remark}

\begin{remark}
Fujiki~\cite{fujiki88} gave a complete classification of finite groups acting on complex tori of dimension two. The subgroups of $\mathrm{SU}(2)$ that can act are precisely the cyclic groups of order 2, 3, 4, 6, the quaternion group, the binary dihedral group of order 12, and the binary tetrahedral group of order 24 (see Section~3 of~\cite{fujiki88}, particularly Lemma~3.3). For the latter three groups one uses quaternion algebras to construct the action. For example, denoting by $H$ the set of Hurwitz quaternions $H:=\mathbb{Z}[1,i,j,(1+i+j+k)/2]$, the automorphism group of the torus $T:=\mathbb{H}/H$ will be given by the group of units $H^{\times}$, which is isomorphic to the binary tetrahedral group.

Of course, if $G$ is a subgroup of $\mathrm{SU}(2)$ then $T/G$ will admit a crepant resolution, so Fujiki's classification yields all K3 surfaces of Kummer type. However, only the cyclic groups yield isotrivial elliptic K3 surfaces, and only for the complex tori and actions described above. In the other cases, the group actions do not preserve any fibration $T\rightarrow E$ on the torus $T$.
\end{remark}

Next we consider to what extent these configurations of singular fibres determine the K3 surface.

\begin{proposition}
Let $S\rightarrow\P^1$ be an isotrivial elliptic K3 surface with four singular fibres of type $I^*_0$. Then $S$ is the Kummer surface constructed from a two-dimensional complex torus $A$ that is fibred over an elliptic curve.
\end{proposition}

\begin{proof}
First observe that $S$ contains sixteen $-2$-curves with multiplicity one, four in each $I^*_0$. Next, the monodromy around each singular fibre is {\em equal\/} to $-\mathrm{I}$ (because every conjugate of $-\mathrm{I}$ is actually equal to $-\mathrm{I}$), and therefore the monodromy group is $\langle -\mathrm{I}\rangle$, of order two. Pulling back the elliptic fibration by the monodromy double cover of the base gives
$$\begin{array}{ccc} \hat{A} & \longrightarrow & S \\
    \downarrow & & \downarrow \\
    E & \longrightarrow & {\P}^1.
\end{array}$$
Here $E$ is the double cover of $\P^1$ branched at the four points corresponding to the singular fibres, and therefore an elliptic curve. The surface $\hat{A}$ has sixteen $-1$-curves covering the sixteen $-2$-curves in $S$ mentioned above; blowing these down yields a surface $A$ which is an isotrivial elliptic fibration over $E$ with no singular fibres. Moreover, the monodromy of $A\rightarrow E$ is trivial, so $A$ is necessarily a two-dimensional complex torus. This completes the proof.
\end{proof}

\begin{corollary}
\label{generic}
Let $S\rightarrow\P^1$ be an isotrivial elliptic K3 surface whose smooth fibres are isomorphic to an elliptic curve $F$ with $j(F)\neq 0$ or $1728$, i.e., $F\not\cong\mathbb{C}/\langle 1,\zeta\rangle$ or $\mathbb{C}/\langle 1,i\rangle$. Then $S$ is the Kummer surface constructed from a two-dimensional complex torus $A$ that is fibred by elliptic curves isomorphic to $F$.
\end{corollary}

\begin{proof}
The automorphism group of $F$ is $\langle -1\rangle$, so the monodromy around any singular fibre of $S\rightarrow\P^1$ must have order two. It follows that the only possible singular fibres are of type $I^*_0$, and the Euler number forces there to be four of them. The result now follows from the previous proposition.
\end{proof}

\begin{proposition}
Let $S\rightarrow\P^1$ be an isotrivial elliptic K3 surface with three singular fibres of type $IV^*$. Then $S$ is the K3 surface of Kummer type given by resolving $E\times E/(\mathbb{Z}/3\mathbb{Z})$, where $E$ is the elliptic curve $\mathbb{C}/\langle 1,\zeta\rangle$, as in the second example above.
\end{proposition}

\begin{proof}
The $j$-function vanishes at a singular fibre of type $IV^*$, so the $j$-function of $S\rightarrow\P^1$ necessarily vanishes identically, i.e., the smooth fibres are isomorphic to $\mathbb{C}/\langle 1,\zeta\rangle$. The monodromy around each singular fibre is conjugate to $\alpha^4=\left(\begin{array}{cc} -1 & -1 \\ 1 & 0 \end{array}\right)$. By an appropriate choice of basis, we can assume that one of these monodromies is {\em equal\/} to $\alpha^4$. Since the product of the three monodromies must be trivial, the lemma below implies that the monodromies are {\em all equal\/} to $\alpha^4$, and the monodromy group is therefore $\langle\alpha^4\rangle=\langle\alpha^2\rangle$, of order three. Pulling back the elliptic fibration by the monodromy triple cover of the base gives
$$\begin{array}{ccc} \hat{A} & \longrightarrow & S \\
    \downarrow & & \downarrow \\
    E & \longrightarrow & {\P}^1.
\end{array}$$
Here $E$ is the triple cover of $\P^1$ branched at the three points corresponding to the singular fibres, and therefore also isomorphic to the elliptic curve $\mathbb{C}/\langle 1,\zeta\rangle$. As in the previous proposition, the surface $\hat{A}$ contains $-1$-curves which can be blown down to yield a surface $A$ that is an isotrivial elliptic fibration over $E$ with no singular fibres. Moreover, the monodromy of $A\rightarrow E$ is trivial, so $A$ is necessarily a two-dimensional complex torus. Since the cyclic group $\langle\alpha^2\rangle$ of order three acts and preserves the fibration $A\rightarrow E$, $A$ must be isomorphic to $E\times E$, as we saw earlier. This completes the proof.
\end{proof}

\begin{lemma}
Let $u\alpha^4u^{-1}$ and $v\alpha^4v^{-1}$ be two conjugates of $\alpha^4$ in $\mathrm{SL}(2,\mathbb{Z})$ such that
$$\alpha^4 u\alpha^4u^{-1}v\alpha^4v^{-1}=\mathrm{I}.$$
Then $u=\pm\mathrm{I}$ and $v=\pm\mathrm{I}$, and thus $u\alpha^4u^{-1}=\alpha^4$ and $v\alpha^4v^{-1}=\alpha^4$.
\end{lemma}

\begin{proof}
Recall that $\alpha=\left(\begin{array}{cc} 1 & 1 \\ -1 & 0 \end{array}\right)$ and $\beta=\left(\begin{array}{cc} 0 & 1 \\ -1 & 0 \end{array}\right)$ generate $\mathrm{SL}(2,\mathbb{Z})$, with
$$\mathrm{SL}(2,\mathbb{Z})\cong\mathbb{Z}/6\mathbb{Z}*_{\mathbb{Z}/2\mathbb{Z}}\mathbb{Z}/4\mathbb{Z}.$$
It is easier to work with $\mathrm{PSL}(2,\mathbb{Z})=\mathrm{SL}(2,\mathbb{Z})/\langle -\mathrm{I}\rangle$, which is a free group
$$\mathrm{PSL}(2,\mathbb{Z})\cong\mathbb{Z}/3\mathbb{Z}*\mathbb{Z}/2\mathbb{Z}$$
generated by $\bar{\alpha}$ and $\bar{\beta}$. We rewrite the equation in $\mathrm{PSL}(2,\mathbb{Z})$ as
$$\bar{\alpha}\bar{u}\bar{\alpha}\bar{u}^{-1}=\bar{v}\bar{\alpha}^2\bar{v}^{-1}.$$
We assume that $\bar{u}$ and $\bar{v}$ are reduced words in $\bar{\alpha}$ and $\bar{\beta}$. Moreover, we can reduce the exponents mod 3 and 2, respectively, so that $\bar{\alpha}$ only appears in these words to the exponent 1 or 2 and $\bar{\beta}$ only appears to the exponent 1. Without loss of generality, neither $\bar{u}$ nor $\bar{v}$ end with a power of $\bar{\alpha}$, so the only place where cancelation may occur in $\bar{\alpha}\bar{u}\bar{\alpha}\bar{u}^{-1}$ and $\bar{v}\bar{\alpha}^2\bar{v}^{-1}$ is between $\bar{\alpha}$ and $\bar{u}$ in the first word. Suppose that $\bar{u}$ is non-trivial. We consider three case.

{\bf Case 1:} The word $\bar{u}$ begins with $\bar{\beta}$. Then there is no cancelation between $\bar{\alpha}$ and $\bar{u}$. Moreover, $\bar{u}^{-1}$ ends with $\bar{\beta}^{-1}$. This implies that $\bar{v}^{-1}$ also ends with $\bar{\beta}^{-1}$, i.e., $\bar{v}$ begins with $\bar{\beta}$. But then $\bar{\alpha}\bar{u}\bar{\alpha}\bar{u}^{-1}$ begins with $\bar{\alpha}$ while $\bar{v}\bar{\alpha}^2\bar{v}^{-1}$ begins with $\bar{\beta}$, a contradiction.

{\bf Case 2:} The word $\bar{u}$ begins with $\bar{\alpha}$. Then there is no cancelation between $\bar{\alpha}$ and $\bar{u}$, and $\bar{\alpha}\bar{u}\bar{\alpha}\bar{u}^{-1}$ begins with $\bar{\alpha}^2$. Moreover, $\bar{u}^{-1}$ ends with $\bar{\alpha}^{-1}$. This implies that $\bar{v}^{-1}$ also ends with $\bar{\alpha}^{-1}$, i.e., $\bar{v}$ begins with $\bar{\alpha}$. But then $\bar{\alpha}\bar{u}\bar{\alpha}\bar{u}^{-1}$ begins with $\bar{\alpha}^2$ while $\bar{v}\bar{\alpha}^2\bar{v}^{-1}$ begins with $\bar{\alpha}$, a contradiction.

{\bf Case 3:} The word $\bar{u}$ begins with $\bar{\alpha}^2$. In this case there is cancelation between $\bar{\alpha}$ and $\bar{u}$, and $\bar{\alpha}\bar{u}\bar{\alpha}\bar{u}^{-1}$ begins with $\bar{\beta}$ (recall that $\bar{u}$ does not end with a power of $\bar{\alpha}$, so the initial $\bar{\alpha}^2$ must be followed by a $\bar{\beta}$). Moreover, $\bar{u}^{-1}$ ends with $\bar{\alpha}^{-2}$. This implies that $\bar{v}^{-2}$ also ends with $\bar{\alpha}^{-2}$, i.e., $\bar{v}$ begins with $\bar{\alpha}^2$. But then $\bar{\alpha}\bar{u}\bar{\alpha}\bar{u}^{-1}$ begins with $\bar{\beta}$ while $\bar{v}\bar{\alpha}^2\bar{v}^{-1}$ begins with $\bar{\alpha}^2$, a contradiction.

We conclude that $\bar{u}$ must be trivial. It follows easily that $\bar{v}$ is also trivial. This means that $u$ and $v$ are words in $\alpha^3=-\mathrm{I}$ and $\beta^2=-\mathrm{I}$, so they are both equal to $\pm\mathrm{I}$.
\end{proof}

\begin{proposition}
Let $S\rightarrow\P^1$ be an isotrivial elliptic K3 surface with three singular fibres, two of type $III^*$ and one of type $I^*_0$. Then $S$ is the K3 surface of Kummer type given by resolving $E\times E/(\mathbb{Z}/4\mathbb{Z})$, where $E$ is the elliptic curve $\mathbb{C}/\langle 1,i\rangle$, as in the third example above.
\end{proposition}

\begin{proof}
The $j$-function is equal to 1728 at a singular fibre of type $III^*$, so the $j$-function of $S\rightarrow\P^1$ is necessarily identically equal to 1728, i.e., the smooth fibres are isomorphic to $\mathbb{C}/\langle 1,i\rangle$. The monodromy around the singular fibres of type $III^*$ are conjugate to $\beta^3=\left(\begin{array}{cc} 0 & -1 \\ 1 & 0 \end{array}\right)$. By an appropriate choice of basis, we can assume that the first of these monodromies is {\em equal\/} to $\beta^3$. The monodromy around the singular fibre of type $I^*_0$ is {\em equal\/} to $-\mathrm{I}=\beta^2$. Since the product of the three monodromies must be trivial, the monodromy around the second fibre of type $III^*$ must be equal to
$$(\beta^3\beta^2)^{-1}=\beta^{-1}=\beta^3$$
too, and the monodromy group is therefore $\langle\beta^3\rangle=\langle\beta\rangle$, of order four. Pulling back the elliptic fibration by the monodromy quadruple cover of the base gives
$$\begin{array}{ccc} \hat{A} & \longrightarrow & S \\
    \downarrow & & \downarrow \\
    E & \longrightarrow & {\P}^1.
\end{array}$$
Here $E$ is the quadruple cover of $\P^1$ totally ramified (ramification index 4) at the two points corresponding to the singular fibres of type $III^*$ and with ramification index 2 at the point corresponding to the singular fibre of type $I^*_0$, and therefore also isomorphic to the elliptic curve $\mathbb{C}/\langle 1,i\rangle$. As in the previous propositions, the surface $\hat{A}$ contains $-1$-curves which can be blown down to yield a surface $A$ that is an isotrivial elliptic fibration over $E$ with no singular fibres. Moreover, the monodromy of $A\rightarrow E$ is trivial, so $A$ is necessarily a two-dimensional complex torus, and indeed isomorphic to $E\times E$. This completes the proof.
\end{proof}

\begin{proposition}
Let $S\rightarrow\P^1$ be an isotrivial elliptic K3 surface with three singular fibres, one of type $II^*$, one of type $IV^*$, and one of type $I^*_0$. Then $S$ is the K3 surface of Kummer type given by resolving $E\times E/(\mathbb{Z}/6\mathbb{Z})$, where $E$ is the elliptic curve $\mathbb{C}/\langle 1,\zeta\rangle$, as in the fourth example above.
\end{proposition}

\begin{proof}
The $j$-function vanishes at singular fibres of type $II^*$ and $IV^*$, so the $j$-function of $S\rightarrow\P^1$ necessarily vanishes identically, i.e., the smooth fibres are isomorphic to $\mathbb{C}/\langle 1,\zeta\rangle$. The monodromy around the singular fibre of type $II^*$ is conjugate to $\alpha^5=\left(\begin{array}{cc} 0 & -1 \\ 1 & 1 \end{array}\right)$. By an appropriate choice of basis, we can assume that this monodromy is {\em equal\/} to $\alpha^5$. The monodromy around the singular fibre of type $I^*_0$ is {\em equal\/} to $-\mathrm{I}=\alpha^3$. Since the product of the three monodromies must be trivial, the monodromy around the singular fibre of type $IV^*$ must be equal to
$$(\alpha^5\alpha^3)^{-1}=\alpha^{-2}=\alpha^4,$$
and the monodromy group is therefore $\langle\alpha\rangle$, of order six. Pulling back the elliptic fibration by the monodromy sextuple cover of the base gives
$$\begin{array}{ccc} \hat{A} & \longrightarrow & S \\
    \downarrow & & \downarrow \\
    E & \longrightarrow & {\P}^1.
\end{array}$$
Here $E$ is the sextuple cover of $\P^1$ totally ramified at the point corresponding to the singular fibre of type $II^*$ and with ramification indices 3 and 2, respectively, at the points corresponding to the singular fibres of type $IV^*$ and $I^*_0$, and therefore also isomorphic to the elliptic curve $\mathbb{C}/\langle 1,\zeta\rangle$. As in the previous propositions, the surface $\hat{A}$ contains $-1$-curves which can be blown down to yield a surface $A$ that is an isotrivial elliptic fibration over $E$ with no singular fibres. Moreover, the monodromy of $A\rightarrow E$ is trivial, so $A$ is necessarily a two-dimensional complex torus, and indeed isomorphic to $E\times E$. This completes the proof.
\end{proof}

\subsection{Weierstrass models}

So far we have only looked at elliptic K3 surfaces whose singular fibres are of starred type ($I^*_0$, $II^*$, $III^*$, and $IV^*$). Corollary~\ref{generic} gave a description of all isotrivial elliptic K3 surfaces with $j$-function not equal to 0 or 1728; in particular, they only have singular fibres of type $I^*_0$. So to obtain the other kinds of singular fibres we must look at elliptic surfaces with $j$-function equal to 0 or 1728.

\begin{lemma}
Let $S\rightarrow\P^1$ be an isotrivial elliptic K3 surface with $j$-function equal to 0, that admits a section. Then $S$ is described by the cubic equation
$$y^2z=x^3+b(t)z^3$$
inside the $\P^2$-bundle $\P(\O(4)\oplus\O(6)\oplus\O)$ over $\P^1$, where $b(t)$ is a section of $\O(12)$ that does not have a zero of order six or greater.
\end{lemma}

\begin{proof}
This is just the Weierstrass model of $S$. In general, an elliptic K3 surface that admits a section will be birational to a family of plane cubics. More precisely, if we blow down the components of the singular fibres of $S\rightarrow\P^1$ that do not meet the section, thereby creating a surface $\overline{S}$ with rational double point singularities, then $\overline{S}$ will be a family of cubics inside the $\P^2$-bundle $\P(\O(4)\oplus\O(6)\oplus\O)$ over $\P^1$. If we let $(x,y,z)$ denote fibre coordinates on $\O(4)\oplus\O(6)\oplus\O$, and $t$ denote a local coordinate on the base $\P^1$, then these cubics will be given by an equation
$$y^2z=x^3+a(t)xz^2+b(t)z^3,$$
where $a(t)$ and $b(t)$ are sections of $\O(8)$ and $\O(12)$, respectively. The $j$-function will be given by
$$1728\frac{4a(t)^3}{4a(t)^3+27b(t)^2}.$$
Since we want this to equal 0, $a(t)$ must vanish identically, giving the required equation. Finally, if $b(t)$ has a zero of order $m$ then $\overline{S}$ locally looks like $y^2=x^3+t^m$. For $m\geq 6$, this would be a singularity that is worse than a rational double point, which is impossible (below we analyse the singularities that arise for $m\leq 5$).
\end{proof}

When $b(t)\neq 0$, the cubic is isomorphic to $y^2=x^3+1$ after rescaling (written in a coordinate patch where $z=1$), which is the equation of $\mathbb{C}/\langle 1,\zeta\rangle$. Singular fibres appear when $b(t)$ has a zero, as described in the table below.

\begin{center}
\begin{tabular}{|c|c|c|c|}
  \hline
  order of zero & Kodaira type & Euler number & monodromy \\
  \hline
  1 & $II$ cusp & 2 & $\alpha$ \\
  2 & $IV$ triple point & 4 & $\alpha^2$ \\
  3 & $I^*_0$ & 6 & $\alpha^3$ \\
  4 & $IV^*$ & 8 & $\alpha^4$ \\
  5 & $II^*$ & 10 & $\alpha^5$ \\
  \hline
\end{tabular}
\end{center}

Note that as simple zeros collide to give zeros of higher order, the Euler numbers add and the monodromies multiply. Also, the surface $\overline{S}$ is smooth only for simple zeros. If $b(t)$ has a zero of order two, for example, then $\overline{S}$ will locally look like $y^2=x^3+t^2$, which is an $A_2$ singularity; the singular fibre of type $IV$ appears after we blow up this singularity. If $b(t)$ has a zero of order three, then $\overline{S}$ will locally look like $y^2=x^3+t^3$, which is a $D_4$ singularity (a change of variables gives the more familiar form $y^2=x(x^2+t^2)$); the singular fibre of type $I^*_0$ appears after blowing up. If $b(t)$ has a zero of order four, then $\overline{S}$ will locally look like $y^2=x^3+t^4$, which is an $E_6$ singularity (quotient of $\mathbb{C}^2$ by the binary tetrahedral group); the singular fibre of type $IV^*$ appears after blowing up. If $b(t)$ has a zero of order five, then $\overline{S}$ will locally look like $y^2=x^3+t^5$, which is an $E_8$ singularity (quotient of $\mathbb{C}^2$ by the binary icosahedral group); the singular fibre of type $II^*$ appears after blowing up. Finally, $b(t)$ cannot have a zero of order six or greater, because the singularities of $\overline{S}$ can be no worse than rational double points. This applies also at $t=\infty$ (if the leading term of $b(t)$ has degree $12-m$ then $t=\infty$ is a zero of order $m$).

\begin{question}
The monodromy around a simple zero must be a conjugate of $\alpha$. Can we say anything stronger than this? For example, if there are twelve simple zeros, are the monodromies all {\em equal\/} to $\alpha$, or is it possible to find twelve different conjugates of $\alpha$ whose product is still $\mathrm{I}$? One could pose similar questions for zeros of higher order.
\end{question}

\begin{lemma}
Let $S\rightarrow\P^1$ be an isotrivial elliptic K3 surface with $j$-function equal to 1728, that admits a section. Then $S$ is described by the cubic equation
$$y^2z=x^3+a(t)xz^2$$
inside the $\P^2$-bundle $\P(\O(4)\oplus\O(6)\oplus\O)$ over $\P^1$, where $a(t)$ is a section of $\O(8)$ that does not have a zero of order four or greater.
\end{lemma}

\begin{proof}
The argument is identical to the one above, except that now
$$1728\frac{4a(t)^3}{4a(t)^3+27b(t)^2}=1728$$
forces $b(t)$ to vanish identically.
\end{proof}

When $a(t)\neq 0$, the cubic is isomorphic to $y^2=x^3+x$ after rescaling (written in a coordinate patch where $z=1$), which is the equation of $\mathbb{C}/\langle 1,i\rangle$. Singular fibres appear when $a(t)$ has a zero, as described in the table below.

\begin{center}
\begin{tabular}{|c|c|c|c|}
  \hline
  order of zero & Kodaira type & Euler number & monodromy \\
  \hline
  1 & $III$ tacnode & 3 & $\beta$ \\
  2 & $I^*_0$ & 6 & $\beta^2$ \\
  3 & $III^*$ & 9 & $\beta^3$ \\
  \hline
\end{tabular}
\end{center}

Again, as simple zeros collide to give zeros of higher order, the Euler numbers add and the monodromies multiply. Also, the surface $\overline{S}$ is never smooth above a zeros of $a(t)$. At a simply zero of $a(t)$, for example, $\overline{S}$ will locally look like $y^2=x^3+tx$, which is an $A_1$ singularity; the singular fibre of type $III$ appears after a blow up. If $a(t)$ has a zero of order two, then $\overline{S}$ will locally look like $y^2=x^3+t^2x$, which is a $D_4$ singularity; the singular fibre of type $I^*_0$ appears after blowing up. If $a(t)$ has a zero of order three, then $\overline{S}$ will locally look like $y^2=x^3+t^3x$, which is an $E_7$ singularity (quotient of $\mathbb{C}^2$ by the binary octahedral group); the singular fibre of type $III^*$ appears after blowing up. Finally, $a(t)$ cannot have a zero of order four or greater; otherwise $\overline{S}$ would have a singularity that is worse than a rational double point, which is impossible. This applies also at $t=\infty$ (if the leading term of $a(t)$ has degree $8-m$ then $t=\infty$ is a zero of order $m$).

\begin{remark}
We can now apply the general theory of elliptic surfaces, as developed by Kodaira~\cite{kodaira63}, to describe isotrivial elliptic K3 surfaces that do not admit sections. Namely, one associates to each such surface $S\rightarrow\P^1$ its relative Jacobian $J\rightarrow\P^1$. Then $J\rightarrow\P^1$ admits a section and is the compactification of a principal homogeneous space $J^{\sharp}\rightarrow\P^1$, while $S\rightarrow\P^1$ is the compactification of a torsor $S^{\sharp}\rightarrow\P^1$ over $J^{\sharp}\rightarrow\P^1$. These compactified torsors $S\rightarrow\P^1$ over a fixed $J\rightarrow\P^1$ are parametrized by the Tate-Shafarevich group of $J$, which in this case is a connected topological group of complex dimension one. See Barth et al.~\cite{bhpv04} for details.
\end{remark}

\section{Lagrangian fibrations in higher dimensions}

\subsection{Hilbert schemes of points on K3 surfaces}

Having described isotrivial elliptic K3 surfaces in the previous section, we can use them to construct isotrivial Lagrangian fibrations in higher dimensions.

\begin{lemma}
Let $S\rightarrow\P^1$ be an isotrivial elliptic K3 surface. Then the Hilbert scheme $\mathrm{Hilb}^nS$ of $n$ points on $S$ admits an isotrivial Lagrangian fibration.
\end{lemma}

\begin{proof}
It is well known that the fibration $S\rightarrow\P^1$ induces a morphism
$$\mathrm{Hilb}^nS\longrightarrow\mathrm{Sym}^nS\longrightarrow\mathrm{Sym}^n\P^1\cong\P^n.$$
Results of Matsushita~\cite{matsushita99} imply that this must be a Lagrangian fibration. Moreover, if the smooth fibres of $S\rightarrow\P^1$ are isomorphic to the elliptic curve $E$, then the smooth fibres of $\mathrm{Hilb}^nS\rightarrow\P^n$ will be isomorphic to $E\times\cdots\times E$, i.e., this will be an isotrivial fibration.
\end{proof}

The Hilbert schemes of K3 surfaces of Kummer type are particularly interesting.

\begin{example}
Let $S$ be a K3 surface of Kummer type given by the resolution of $E\times E/G$ where $G$ is the cyclic group of order 2, 3, 4, or 6, and $E$ is an elliptic curve containing $G$ in its group of automorphisms. We saw in Section~2.2 that $S$ admits an isotrivial elliptic fibration, and therefore $\mathrm{Hilb}^nS$ admits an isotrivial Lagrangian fibration. Moreover, the composition of morphisms
$$\mathrm{Hilb}^nS\longrightarrow\mathrm{Sym}^nS= (S\times\cdots\times S)/\mathfrak{S}_n\longrightarrow (E\times E\times\cdots\times E\times E)/(G^n\rtimes\mathfrak{S}_n)$$
reveals that $\mathrm{Hilb}^nS$ is a crepant resolution of the orbifold arising from the action of the semi-direct product $G^n\rtimes\mathfrak{S}_n$ on the $2n$-dimensional abelian variety $E^{2n}$. The Lagrangian fibration is induced by the projection onto the even factors, i.e., the second, fourth, $\ldots$, $2n$-th copies of $E$, which of course gives the base
$$(E\times\ldots\times E)/(G^n\rtimes\mathfrak{S}_n)\cong (E/G)^n/\mathfrak{S}_n\cong (\P^1)^n/\mathfrak{S}_n\cong\P^n.$$

Consider the particular case when $n=3$ and $G\cong\mathbb{Z}/2\mathbb{Z}$. For $(x,y)\in E\times E$, the action of $\mathbb{Z}/2\mathbb{Z}$ is given simply by $(x,y)\mapsto (-x,-y)$. Thus the $(\mathbb{Z}/2\mathbb{Z})^3\rtimes\mathfrak{S}_3$ action on $E^6$ is generated by
$$\gamma_1:(x_1,y_1,x_2,y_2,x_3,y_3)\mapsto (-x_1,-y_1,x_2,y_2,x_3,y_3),$$
$$\gamma_2:(x_1,y_1,x_2,y_2,x_3,y_3)\mapsto (x_1,y_1,-x_2,-y_2,x_3,y_3),$$
$$\gamma_3:(x_1,y_1,x_2,y_2,x_3,y_3)\mapsto (x_1,y_1,x_2,y_2,-x_3,-y_3),$$
plus the generators of $\mathfrak{S}_3$ acting by permuting $(x_1,y_1)$, $(x_2,y_2)$, and $(x_3,y_3)$. Later we will describe a new example by modifying this action.
\end{example}

\subsection{Generalized Kummer varieties}

\begin{lemma}
Let $A\rightarrow E$ be a two-dimensional complex torus that is fibred over an elliptic curve $E$. Then the generalized Kummer variety $K_n(A)$ admits an isotrivial Lagrangian fibration.
\end{lemma}

\begin{proof}
Recall the construction of the generalized Kummer variety (see Beauville~\cite{beauville83}): we take the Hilbert scheme $\mathrm{Hilb}^{n+1}A$ of $n+1$ points on $A$, and then compose the Hilbert-Chow morphism with the morphism given by adding points in $A$ to get
$$\mathrm{Hilb}^{n+1}A\longrightarrow\mathrm{Sym}^{n+1}A\longrightarrow A.$$
The generalized Kummer variety $K_n(A)$ is the fibre of this morphism over zero (in fact, all fibres are isomorphic).

Now the fibration $A\rightarrow E$ induces a morphism
$$\mathrm{Hilb}^{n+1}A\longrightarrow\mathrm{Sym}^{n+1}A\longrightarrow\mathrm{Sym}^{n+1}E\cong\P^n\times E$$
that induces a morphism $K_n(A)\rightarrow\P^n$. Once again, results of Matsushita~\cite{matsushita99} imply that this must be a Lagrangian fibration. If the fibres of $A\rightarrow E$ are isomorphic to the elliptic curve $F$, then the smooth fibres of $K_n(A)\rightarrow\P^n$ will be isomorphic to the $n$-dimensional abelian variety
$$\{(x_1,\ldots,x_{n+1})\in F^{n+1}\;|\;x_1+\ldots+x_{n+1}=0\}\subset F^{n+1},$$
showing that this fibration is isotrivial.
\end{proof}

\begin{remark}
As in the previous subsection, the generalized Kummer variety $K_n(A)$ is a crepant resolution of the orbifold arising from the action of $\mathfrak{S}_{n+1}$ on the $2n$-dimensional abelian variety
$$\{(x_1,y_1,\ldots,x_{n+1},y_{n+1})\in A^{n+1}\;|\;(x_1,y_1)+\ldots+(x_{n+1},y_{n+1})=(0,0)\}\subset A^{n+1}.$$
The Lagrangian fibration is induced by the projection of each factor $A$ onto $E$, which gives the base
$$\{(y_1,\ldots,y_{n+1})\in E^{n+1}\;|\;y_1+\ldots+y_{n+1}=0\}/\mathfrak{S}_{n+1}\cong \P^n,$$
as this is a complete linear system on $E$ of degree $n+1$.
\end{remark}

\subsection{A singular example}

The examples of the previous subsections can be described as crepant resolutions of orbifolds arising from actions of finite groups on complex tori. This construction is reminiscent of Joyce's construction of compact Riemannian manifolds with special holonomy~\cite{joyce00}. He takes a finite subgroup $\Gamma$ of $G_2\subset\mathrm{SO}(7)$ whose action on $\mathbb{R}^7$ preserves the lattice $\mathbb{Z}^7$, so that $\Gamma$ acts on the torus $T=\mathbb{R}^7/\mathbb{Z}^7$. Then he resolves the singularities of $T/\Gamma$. Since $\Gamma\subset G_2$, there is a $G_2$-structure (a three-form) on the resolved space $M$ induced by the standard $G_2$-structure on $\mathbb{R}^7$. However, proving that $M$ admits a Riemannian metric with holonomy $G_2$ relies on some difficult analysis (see~\cite{joyce00}). Joyce also produces examples of manifolds with holonomy $\mathrm{Spin}(7)$ in this way.

To construct a compact hyperk{\"a}hler manifold (aka holomorphic symplectic manifold) we should start with a finite subgroup $\Gamma$ of $\mathrm{Sp}(n)\subset\mathrm{SO}(4n)$ whose action on $\mathbb{C}^{2n}$ preserves a lattice $\Lambda$ of rank $4n$, so that $\Gamma$ acts (holomorphic symplectically) on the complex torus $T=\mathbb{C}^{2n}/\Lambda$. If $T/\Gamma$ admits a symplectic desingularization (equivalent to a crepant resolution in this case) then the resolved space $M$ will be a holomorphic symplectic manifold. Because $M$ will be compact and K{\"a}hler, it will automatically admit a hyperk{\"a}hler metric by Yau's Theorem, i.e., the difficult analysis of the $G_2$ and $\mathrm{Spin}(7)$ cases can be avoided. Thus the main challenge is to find appropriate actions of finite groups $\Gamma\subset\mathrm{Sp}(n)$ on complex tori $T$. Here we describe an example leading to a holomorphic symplectic orbifold $T/\Gamma$, which unfortunately does not admit a symplectic desingularization.

In all of this, the existence of an isotrivial Lagrangian fibration is somewhat peripheral. In fact, it only serves to constrain the problem, but this allows us to focus our search for appropriate group actions, which could otherwise become unwieldy. Thus we consider tori of the form $T=A\times B$ where $A$ and $B$ are $n$-dimensional complex tori, such that the projection onto $B$ induces
$$T/\Gamma\longrightarrow B/\Gamma$$
with $B/\Gamma\cong\P^n$. Note that $\P^n$ arose in this way in the examples of the previous subsections.

\begin{question}
Aside from the examples
$$(E\times\cdots \times E)/(G^n\rtimes\mathfrak{S}_n)\cong\P^n$$
and
$$\{(y_1,\ldots,y_{n+1})\in E^{n+1}\;|\;y_1+\ldots+y_{n+1}=0\}/\mathfrak{S}_{n+1}\cong\P^n,$$
are there other ways in which $\P^n$ can arise as a quotient $B/\Gamma$ of an $n$-dimensional complex torus by a finite group?
\end{question}

Starting with one of the examples above, with $B/\Gamma\cong\P^n$, we need to extend the $\Gamma$-action to $T=A\times B$. The fact that this action must be symplectic almost completely determines how the group acts on $A$: the only freedom remaining is to include some translations. Again, this is highly reminiscent of Joyce's actions. The following example in six dimensions, produced by modifying the action at the end of Subsection~3.1, illustrates this approach.

\begin{proposition}
Let the generators of $\Gamma=(\mathbb{Z}/2\mathbb{Z})^3\rtimes\mathfrak{S}_3$ act on $E^6$ by
$$\gamma_1:(x_1,y_1,x_2,y_2,x_3,y_3)\mapsto \left(-x_1,-y_1,x_2+\frac{1}{2},y_2,x_3+\frac{1}{2},y_3\right),$$
$$\gamma_2:(x_1,y_1,x_2,y_2,x_3,y_3)\mapsto \left(x_1+\frac{1}{2},y_1,-x_2,-y_2,x_3+\frac{1}{2},y_3\right),$$
$$\gamma_3:(x_1,y_1,x_2,y_2,x_3,y_3)\mapsto \left(x_1+\frac{1}{2},y_1,x_2+\frac{1}{2},y_2,-x_3,-y_3\right),$$
with the generators of $\mathfrak{S}_3$ acting by permuting $(x_1,y_1)$, $(x_2,y_2)$, and $(x_3,y_3)$ as before. (We use $\frac{1}{2}$ here to denote any $2$-torsion point of $E$.) This determines an action of $\Gamma$ on $E^6$. The quotient $E^6/\Gamma$ is a holomorphic symplectic orbifold that admits an isotrivial Lagrangian fibration over $\P^3$.
\end{proposition}

\begin{remark}
Recall that a {\em singular symplectic variety\/} in the sense of Beauville~\cite{beauville00} is a normal variety whose smooth part admits a symplectic two-form that extends to a (possibly degenerate) holomorphic two-form on any resolution of singularities. Holomorphic symplectic orbifolds are always singular symplectic varieties because quotienting by a finite group that preserves the symplectic form produces symplectic singularities by Proposition~2.4 of~\cite{beauville00}. By a direct calculation of the $\Gamma$-invariant forms on $E^6$, we see that $h^{2,0}(E^6/\Gamma)=1$, showing that $E^6/\Gamma$ is also a {\em primitively symplectic V-manifold\/}, in the terminology of Fujiki~\cite{fujiki83}.
\end{remark}

\begin{proof}
It is straightforward to check that the actions of $\gamma_i$ and $\gamma_j$ commute, and they interact with the permutations of $\mathfrak{S}_3$ in the appropriate way. So we do indeed have an action of $\Gamma=(\mathbb{Z}/2\mathbb{Z})^3\rtimes\mathfrak{S}_3$ on $E^6$. We observe that this action preserves the holomorphic symplectic structure
$$dx_1\wedge dy_1+dx_2\wedge dy_2+dx_3\wedge dy_3$$
on $E^6$, because it is a modification of the action of Subsection~3.1 by translations (or we can see this by direct observation). Therefore the holomorphic symplectic structure descends to the orbifold $E^6/\Gamma$. Finally, the projection from $E^6$ onto its even factors, i.e., the second, fourth, and sixth copies of $E$, induces the isotrivial Lagrangian fibration
$$E^6/\Gamma\longrightarrow E^3/\Gamma\cong\P^3.$$
\end{proof}

Unfortunately there is no symplectic desingularization.

\begin{lemma}
The holomorphic symplectic orbifold $E^6/\Gamma$ constructed above does not admit a symplectic desingularization.
\end{lemma}

\begin{proof}
Fixed points of the $\Gamma$-action on $E^6$ produce singularities of $E^6/\Gamma$. Now $\gamma_1$, $\gamma_2$, and $\gamma_3$ do not fix any points, but their products do fix points. For example, the set of fixed points of
$$\gamma_1\gamma_2:(x_1,y_1,x_2,y_2,x_3,y_3)\mapsto \left(-x_1+\frac{1}{2},-y_1,-x_2+\frac{1}{2},-y_2,x_3,y_3\right)$$
(where we have used the fact that $\frac{1}{2}$ is $2$-torsion) includes points of the form
$$\left(\frac{1}{4},0,\frac{3}{4},0,x_3,y_3\right).$$
For $(x_3,y_3)\neq \left(\frac{1}{4},0\right)$ and $\neq \left(\frac{3}{4},0\right)$, the above point is fixed only by $\gamma_1\gamma_2$. So we obtain a singularity in $E^6/\Gamma$ that locally looks like
$$(\mathbb{C}^4/\pm 1)\times\mathbb{C}^2.$$
It is well-known that $\mathbb{C}^4/\pm 1$ does not admit a symplectic desingularization (for example, see Corollary~3.5 in the survey article by Fu~\cite{fu06}).
\end{proof}

This example admits the following higher-dimensional generalization.

\begin{proposition}
Let the generators of $\Gamma=(\mathbb{Z}/2\mathbb{Z})^n\rtimes\mathfrak{S}_n$ act on $E^{2n}$ by
$$\gamma_1:(x_1,y_1,x_2,y_2,\ldots,x_n,y_n)\mapsto \left(-x_1,-y_1,x_2+\frac{1}{2},y_2,\ldots,x_n+\frac{1}{2},y_n\right),$$
$$\vdots$$
$$\gamma_n:(x_1,y_1,x_2,y_2,\ldots,x_n,y_n)\mapsto \left(x_1+\frac{1}{2},y_1,x_2+\frac{1}{2},y_2,\ldots,-x_n,-y_n\right),$$
with the generators of $\mathfrak{S}_n$ acting by permuting $(x_1,y_1)$, $(x_2,y_2),\ldots,(x_3,y_3)$ as before. This determines an action of $\Gamma$ on $E^{2n}$. The quotient $E^{2n}/\Gamma$ is a holomorphic symplectic orbifold that admits an isotrivial Lagrangian fibration over $\P^n$.
\end{proposition}

\begin{proof}
As before, a straightforward check shows that the actions of $\gamma_i$ and $\gamma_j$ commute, and interact with the permutations of $\mathfrak{S}_n$ to give an action of $\Gamma=(\mathbb{Z}/2\mathbb{Z})^n\rtimes\mathfrak{S}_n$ on $E^{2n}$. This action preserves the holomorphic symplectic structure
$$dx_1\wedge dy_1+dx_2\wedge dy_2+\ldots +dx_n\wedge dy_n$$
on $E^{2n}$, so the holomorphic symplectic structure descends to the orbifold $E^{2n}/\Gamma$. Finally, the projection from $E^{2n}$ onto its even factors induces the isotrivial Lagrangian fibration
$$E^{2n}/\Gamma\longrightarrow E^n/\Gamma\cong\P^n.$$
\end{proof}

\begin{lemma}
The holomorphic symplectic orbifold $E^{2n}/\Gamma$ constructed above does not admit a symplectic desingularization.
\end{lemma}

\begin{proof}
As before, $\gamma_1$, $\gamma_2,\ldots,\gamma_n$ themselves have no fixed points, but, for example,
$$\gamma_1\gamma_2:(x_1,y_1,x_2,y_2,\ldots,x_n,y_n)\mapsto \left(-x_1+\frac{1}{2},-y_1,-x_2+\frac{1}{2},-y_2,x_3,y_3\ldots,x_n,y_n\right)$$
has fixed points of the form
$$\left(\frac{1}{4},0,\frac{3}{4},0,x_3,y_3\ldots,x_n,y_n\right).$$
For $(x_3,y_3),\ldots,(x_n,y_n)$ distinct and not equal to $\left(\frac{1}{4},0\right)$ or $\left(\frac{3}{4},0\right)$, the above point is fixed only by $\gamma_1\gamma_2$. So we obtain a singularity in $E^{2n}/\Gamma$ that locally looks like
$$(\mathbb{C}^4/\pm 1)\times\mathbb{C}^{2n-4}.$$
As before, $\mathbb{C}^4/\pm 1$ does not admit a symplectic desingularization (Corollary~3.5 of Fu~\cite{fu06}).
\end{proof}

\subsection{Matsushita's example}

Another approach to constructing holomorphic symplectic orbifolds that admit isotrivial Lagrangian fibrations is to start with one of the examples $T/\Gamma\rightarrow B/\Gamma$ above and then replace $\Gamma$ by a proper subgroup $\Gamma^{\prime}$. Matushita~\cite{matsushita01} constructed a six-dimensional example in this way.

\begin{example}
Let $G$ be the cyclic group of order $k=2$, $3$, $4$, or $6$ (thought of as $k$th roots of unity in $\mathbb{C}$), let $E$ be an elliptic curve containing $G$ in its group of automorphisms, and let $S=\widehat{E\times E/G}$ be the corresponding K3 surface of Kummer type. We saw that $\mathrm{Hilb}^nS$ is a crepant resolution of $E^{2n}/\Gamma$, where $\Gamma$ is the semi-direct product $G^n\rtimes\mathfrak{S}_n$. Instead, let us take the subgroup
$$\Gamma^{\prime}=\langle (a_1,\ldots,a_n)\in G^n,\tau\in\mathfrak{A}_n\;|\;a_1\cdots a_n=1\rangle<\Gamma,$$
where $\mathfrak{A}_n$ is the alternating group. Of course $\Gamma^{\prime}$ acts on $E^{2n}$ in the usual way, with $(a_1,\ldots,a_n)$ acting by $\mathrm{diag}(a_1,a_1^{-1},\ldots,a_n,a_n^{-1})$ and $\tau\in\mathfrak{A}_n<\mathfrak{S}_n$ acting as a permutation matrix on $E^{2n}=(E^2)^n$. Then $E^{2n}/\Gamma^{\prime}$ is a holomorphic symplectic orbifold and it admits an isotrivial Lagrangian fibration over $E^n/\Gamma^{\prime}$, induced by projection onto the even factors. Matsushita's example~\cite{matsushita01} corresponds to the case $k=6$ and $n=3$.
\end{example}

\begin{lemma}
Let $Z=E^{2n}/\Gamma^{\prime}$. Then 
\begin{enumerate}
\item $h^{1,0}(Z)=0$,
\item if $n\geq 3$ then $h^{2,0}(Z)=1$, so $Z$ is a primitively symplectic V-manifold, cf.~Fujiki~\cite{fujiki83},
\item the base of the Lagrangian fibration on $Z$ is a Calabi-Yau $n$-fold,
\item $Z$ does not admit a symplectic desingularization.
\end{enumerate}
\end{lemma}

\begin{proof}
The holomorphic forms on $Z$ are precisely the $\Gamma^{\prime}$-invariant forms on $E^{2n}$. Statements 1 and 2 then follow from a direct calculation. Moreover, if $(y_1,\ldots,y_n)$ are the `even' coordinates then $dy_1\wedge\cdots\wedge dy_n$ is an $\Gamma^{\prime}$-invariant form on the product $E^n$ of the even factors (indeed, it is the only $\Gamma^{\prime}$-invariant form on $E^n$). This shows that the base $E^n/\Gamma^{\prime}$ is a Calabi-Yau $n$-fold, proving 3. Finally, if $n=2$ then $Z$ has an isolated singularity at $(0,0,0,0)$; there does not exist a symplectic desingularization by Corollary~3.5 of Fu~\cite{fu06}. If $n\geq 3$ then $Z$ locally looks like $(\C^4/G)\times \C^{2n-4}$ around the singularity $(0,0,0,0,x_3,y_3,\ldots,x_n,y_n)$, where $x_3,y_3,\ldots,x_n,y_n$ are generic; again, $\C^4/G$ does not admit a symplectic desingularization by Corollary~3.5 of Fu~\cite{fu06}. This completes the proof of 4.
\end{proof}

\begin{remark}
The point of Matushita's example was to show that a singular holomorphic symplectic variety could be fibred over a normal base $B$ with $K_B$ trivial. By contrast, the base $B$ of a Lagrangian fibration on a smooth holomorphic symplectic manifold must be $\mathbb{Q}$-Fano, i.e., $-K_B$ must be ample~\cite{matsushita99}. If in addition the base is smooth and the total space is projective, then Hwang~\cite{hwang08} proved that the base must be $\P^n$. There are no examples known where the total space is smooth and projective and the base is singular.
\end{remark}

\begin{flushleft}
Department of Mathematics\hfill sawon@email.unc.edu\\
University of North Carolina\hfill www.unc.edu/$\sim$sawon\\
Chapel Hill NC 27599-3250\\
USA\\
\end{flushleft}


\begin{thebibliography}{XXX}





\bibitem{beauville83} A.\ Beauville,
{\em Vari{\'e}t{\'e}s K{\"a}hleriennes dont la premi{\`e}re classe de
Chern est nulle\/},
Jour.\ Diff.\ Geom.\ {\bf 18} (1983), 755--782.



\bibitem{beauville00} A.\ Beauville,
{\em Symplectic singularities\/},
Invent.\ Math.\ {\bf 139} (2000), no.\ 3, 541--549.


\bibitem{bhpv04} W.\ Barth, K.\ Hulek, C.\ Peters, and A.\ Van de Ven,
{\em Compact complex surfaces, 2nd ed.\/},
Ergebnisse der Mathematik (3) {\bf 4}, Springer, Berlin, 2004.

















\bibitem{fu06} B.\ Fu,
{\em A survey on symplectic singularities and symplectic resolutions\/},
Ann.\ Math.\ Blaise Pascal {\bf 13} (2006), no.\ 2, 209--236.

\bibitem{fujiki83} A.\ Fujiki,
{\em On primitively symplectic compact K{\"a}hler $V$-manifolds of
dimension four\/},
in Classification of algebraic and analytic manifolds (Katata, 1982), 71--250, Progr.\ Math.\ {\bf 39}, Birkh{\"a}user Boston, Boston, MA, 1983.

\bibitem{fujiki88} A.\ Fujiki,
{\em Finite automorphism groups of complex tori of dimension two\/},
Publ.\ Res.\ Inst.\ Math.\ Sci.\ {\bf 24} (1988), no.\ 1, 1--97. 
























\bibitem{hwang08} J.-M.\ Hwang,
{\em Base manifolds for fibrations of projective irreducible symplectic manifolds\/},
Invent.\ Math.\ {\bf 174} (2008), no.\ 3, 625--644.




\bibitem{joyce00} D.\ Joyce,
{\em Compact manifolds with special holonomy\/},
Oxford Mathematical Monographs, Oxford University Press, Oxford, 2000.




\bibitem{kodaira63} K.\ Kodaira,
{\em On compact analytic surfaces II, III\/},
Ann.\ of Math.\ {\bf 77} (1963), 563--626, and {\bf 78} (1963), 1--40.
















\bibitem{matsushita99} D.\ Matsushita,
{\em On fibre space structures of a projective irreducible symplectic
manifold\/},
Topology {\bf 38} (1999), no.\ 1, 79--83.
Addendum, Topology {\bf 40} (2001), no.\ 2, 431--432.


\bibitem{matsushita01} D.\ Matsuhita,
{\em Fujiki relation on symplectic varieties\/},
preprint {\bf arXiv:math.AG/0109165}.





\bibitem{mp89} R.\ Miranda and U.\ Persson,
{\em Configurations of $I_n$ fibers on elliptic K3 surfaces\/},
Math.\ Z.\ {\bf 201} (1989), no.\ 3, 339--361.

















\bibitem{sawon12} J.\ Sawon,
{\em A finiteness theorem for Lagrangian fibrations\/},
preprint {\bf arXiv:1212.6470}.











\end{thebibliography}
\end{document}